\documentclass[a4paper,11pt,reqno]{amsart}

\usepackage{url}
\usepackage[colorlinks=true, linkcolor=blue, citecolor=ForestGreen, urlcolor=blue]{hyperref}
\usepackage{cite}

\usepackage{latexsym,amsmath,amssymb,amsfonts,amsthm}
\usepackage{mathrsfs}
\usepackage{mathtools}
\usepackage{dsfont}
\usepackage{bbm}
\usepackage{soul}
\usepackage{cancel}
\usepackage[utf8]{inputenc}

\usepackage[usenames, dvipsnames]{color}
\usepackage[usenames, dvipsnames, cmyk]{xcolor}
\usepackage{graphicx}
\usepackage{multirow}
\usepackage{enumerate}
\usepackage{enumitem}
\usepackage{subcaption}

\usepackage{a4wide}

\allowdisplaybreaks
\numberwithin{equation}{section}

\definecolor{grey}{rgb}{.7,.7,.7}
\definecolor{refkey}{gray}{.45}
\definecolor{labelkey}{gray}{.45}

\newtheorem{theorem}{Theorem}[section]
\newtheorem{proposition}[theorem]{Proposition}
\newtheorem{lemma}[theorem]{Lemma}

\theoremstyle{remark}
\newtheorem{remark}[theorem]{Remark}
\theoremstyle{definition}

\newtheorem{example}[theorem]{Example}

\usepackage{tikz}
\usepackage{pgf,pgfplots}
\usetikzlibrary{calc}
\usetikzlibrary{decorations.markings}
\usetikzlibrary{decorations.pathmorphing}
\usetikzlibrary{decorations.shapes}
\usetikzlibrary{shapes,arrows,shapes.geometric,patterns,fadings}
\usetikzlibrary{pgfplots.fillbetween}

\newcommand{\e}{\varepsilon}

\newcommand{\N}{\mathbb N}

\newcommand{\R}{\mathbb R}

\newcommand{\dd}{\,\mathrm{d}}

\newcommand{\Hd}{\mathcal{H}^{d-1}}
\renewcommand{\setminus}{\backslash}

\newcommand{\defeq}{\coloneqq}

\newcommand{\per}{\mathcal{P}}

\newcommand{\en}{\mathcal{F}_{\varepsilon}}
\newcommand{\enm}{\mathcal{F}_{\varepsilon,m}}
\newcommand{\nl}{\mathcal{J}}
\newcommand{\nla}{\mathcal{J}_{-\alpha}}
\newcommand{\nlb}{\mathcal{J}_{\beta}}
\newcommand{\epsglob}{\e_{\mathrm{ball}}}

\newcommand{\ba}{\begin{array}}
\newcommand{\ea}{\end{array}}

\newcommand{\bthm}{\begin{theorem}}
\newcommand{\ethm}{\end{theorem}}
\newcommand{\bprop}{\begin{proposition}}
\newcommand{\eprop}{\end{proposition}}
\newcommand{\blemma}{\begin{lemma}}
\newcommand{\elemma}{\end{lemma}}
\newcommand{\bexmpl}{\begin{example}}
\newcommand{\eexmpl}{\end{example}}

\newcommand{\beqn}{\begin{equation}}
\newcommand{\eeqn}{\end{equation}}
\newcommand{\beqns}{\begin{equation*}}
\newcommand{\eeqns}{\end{equation*}}

\newcommand{\pr}{\prime}
\newcommand{\pt}{\partial}

\newcommand{\ol}{\overline}

\renewcommand{\leq}{\leqslant}
\renewcommand{\geq}{\geqslant}

\definecolor{mygreen}{rgb}{0.1,0.75,0.2}

\newcommand{\veps}{\varepsilon}

\newcounter{myenumi}
\DeclareMathOperator{\loc}{loc}

\title[Attractive-repulsive energies with perimeter penalization]{Stability and minimality of the ball for attractive-repulsive energies with perimeter penalization}

\author{Marco Bonacini}
\address[Marco Bonacini]{Department of Mathematics, University of Trento, Italy}
\email{marco.bonacini@unitn.it}

\author{Ihsan Topaloglu}
\address[Ihsan Topaloglu]{Department of Mathematics and Applied Mathematics, Virginia Commonwealth University, Richmond, VA, USA}
\email{iatopaloglu@vcu.edu}

\date{\today}       

\thanks{This is a post-peer-review, pre-copyedit version of an article published in the Interfaces and Free Boundaries. The final authenticated version is available online at: \url{https://doi.org/10.4171/IFB/548}.}                                 


\begin{document}

\begin{abstract}
We consider perimeter perturbations of a class of attractive-repulsive energies, given by the sum of two nonlocal interactions with power-law kernels, defined over sets with fixed measure. We prove that there exists curves in the perturbation-volume parameters space that separate stability/instability and global minimality/non-minimality regions of the ball, and provide a precise description of these curves for certain interaction kernels. In particular, we show that in small perturbation regimes there are (at least) two disconnected regions for the mass parameter in which the ball is stable, separated by an instability region.  
\end{abstract}

\maketitle



\section{Introduction}\label{sec:intro}

In this note we consider the following minimization problem among sets of finite perimeter $E\subset\R^d$ with fixed volume ($d$-dimensional Lebesgue measure) $|E|=V$:
\begin{equation} \label{eq:min}
\min\Bigl\{ \en(E) \colon E\subset\R^d,\, |E|=V \Bigr\},
\end{equation}
where the functional $\en$ is a perimeter perturbation of a prototypical attractive-repulsive energy defined via power-law kernels
\begin{equation}\label{eq:energy}
\en(E)\defeq \veps\,\per(E)+\int_E\int_E \frac{1}{|x-y|^\alpha}\dd x\dd y + \int_E\int_E |x-y|^\beta\dd x\dd y, \quad \alpha\in(0,d),\,\beta>0.
\end{equation}
Here $\per(E)$ denotes the perimeter of the set $E$ in the sense of Caccioppoli–De Giorgi. For notational convenience we write
\begin{equation}\label{eq:nl_interaction}
\nl_\sigma(E)\defeq \int_E\int_E |x-y|^\sigma\dd x\dd y, \qquad \text{for } \sigma>-d.
\end{equation}
With this notation $\en(E)=\veps\,\per(E)+\nla(E)+\nlb(E)$.

Introduced in \cite{BurChoTop18}, purely nonlocal interaction energies of the form $\nla(E)+\nlb(E)$ arise in descriptions of systems of uniformly distributed interacting particles, in particular in models of collective behavior of many-agent systems related to swarming (see e.g. \cite{BernoffTopaz,FetecauHuangKolokolnikov,TopazBertozzi2}). From a physical point of view, the inclusion of a surface penalization in swarming or aggregation models is rather natural since in most physical and biological systems one would like to minimize the interaction between the bulk and the void. The goal of this note is to characterize the parameter regimes $(V,\e)$ in which the ball is stable or a global minimizer of the functional $\en$.

\medskip

Energies that include various combinations of the terms appearing in \eqref{eq:energy} have been widely studied in the literature. For the purely nonlocal interaction energy $\mathcal{F}_0=\nla+\nlb$, in the large volume regime, in \cite{FraLie21} the authors proved that for all $\beta>0$ and $0<\alpha<d-1$ there is a threshold $V_{\rm ball}\in(0,\infty)$ such that the ball with volume $V$ is the only (up to translation) minimizer of $\mathcal{F}_0$ for $V>V_{\rm ball}$. It is also observed in \cite{FraLie21}, for $\alpha\in[d-1,d)$ balls are never minimizers. Concerning the small volume regime, in \cite{BurChoTop18} it is shown that for $\beta=2$ and $\alpha\in[d-2,d)$ the energy $\mathcal{F}_0$ does not admit a minimizer for $V$ sufficiently small. In \cite{FraLie18} the same is proved for $d=3$, $\alpha=1$, and any $\beta>0$. In a recent joint work with Cristoferi \cite{BonCriTop24}, we investigate the \emph{stability and local optimality of the ball}, and characterize the exact range of volumes for which the ball is a volume-constrained stable set for $\mathcal{F}_0$ (in the sense of nonnegativity of the second variation of the energy with respect to smooth perturbations of the boundary of the ball). Moreover, we prove that the strict stability of the ball yields its (quantitative) local minimality among sets with the same volume and whose boundary is contained in a small uniform neighborhood of the boundary of the ball.

The energy $\per(E)+\nla(E)$, on the other hand, appears in Gamow’s liquid drop model for the atomic nucleus \cite{Ga1930}. As a prototypical geometric variational problem, where the competition between attraction and repulsion creates a dichotomy of existence and nonexistence depending on the volume constraint $|E|=V$, this model has received significant interest in the last two decades (see \cite{ChoMurTop17,Fra} for a review). Indeed, for any $\alpha\in (0,d)$ and for $V$ sufficiently small the energy $\per(E)+\nla(E)$ admits the ball as unique minimizer. On the other hand for $V$ sufficiently large the energy does not admit a minimizer. We refer to \cite{BoCr14,ChoRuo25,FFMMM,FraNam21,KnMu2014} for these results. More relevant to our work, in \cite{BoCr14,FFMMM} the authors investigate the stability of the ball. In particular, they establish a general criterion which states that critical configurations with positive second variation are $L^1$-local minimizers of the energy $\per(E)+\nla(E)$ for $\alpha\in (0,d)$. Moreover these configurations satisfy a stability inequality with respect to the $L^1$-norm. As a result, this criterion provides the existence of a (explicitly determined) critical threshold determining the interval of volumes for which the ball is a local minimizer.


A combination of isoperimetric and attractive-repulsive energies of the form \eqref{eq:energy} (possibly with fractional perimeter) has already been investigated by Ascione in \cite{Asc22}. There the author first proves that for $d\geq 2$, $\alpha\in(0,d)$, $\beta>0$ the problem \eqref{eq:min} admits a minimizer for $\e$ and $V$ sufficiently large, and then establishes that there exists a constant $V_0>0$ and a function $\tilde{\e}(V)$ such that the ball of volume $V$ is the unique (up to translations) minimizer of the problem \eqref{eq:min} for any $\e\geq \tilde{\e}(V)$ and $V>V_0$.

In a recent preprint \cite{RenWanWei}, Ren, Wang, and Wei consider a similar problem in $\R^2$. For $E\subset\R^2$ with $|E|=V$, they study the stability of a single disk for the energy functional
    \[
        \e \, \per(E) - \frac{1}{2\pi}\int_{E}\int_{E} \log|x-y| \dd x \dd y + \int_E \int_E |x-y|^2 \dd x \dd y.
    \]
In particular, they prove that for $V \geq 1/4$, the disk of area $V$ is stable for every $\e>0$ whereas for $V<1/4$ it is stable for $\e >\gamma_V^{-1}(\frac{V}{\pi})^{3/2}$ and unstable for $\e < \gamma_V^{-1}(\frac{V}{\pi})^{3/2}$, where $\gamma_V \defeq \big(\max \big\{ \frac{-1/(2n)+1/2-2V}{n^2-1} \colon n \geq 2 \big\}\big)^{-1}$. Here we establish a similar result for more general power-law kernels in any dimension $d\geq 2$.

\medskip

In the following, it will be convenient to normalize the volume constraint and to work with sets having the fixed mass $|E|=\omega_d$, where $\omega_d\defeq|B_1|$ denotes the volume of the ball in $\R^d$ with radius $1$. By a simple scaling argument this amounts to introducing a parameter $m=\frac{V}{\omega_d}$ in the functional $\en$. Namely,
    \[
        \en \Bigl( \Bigl(\frac{V}{\omega_d}\Bigr)^{1/d}E\Bigr) = \Bigl(\frac{V}{\omega_d}\Bigr)^{(d-1)/d} \enm(E),
    \]
where
\begin{equation}\label{eq:energy-scaled}
\enm(E)\defeq \veps\,\per(E) + m^{(d-\alpha+1)/d}\,\nla(E) + m^{(d+\beta+1)/d} \,\nlb(E)
\end{equation}
for $E\subset \R^d$ with $|E|=\omega_d$. This heuristically shows that for fixed (small) $\e$ we expect two disconnected regions of stability of the ball of volume $V$. Indeed, when $m$ is small, the energy $\enm$ is dominated by the perimeter term for which the ball is the global minimizer whereas for $m$ large it is dominated by the attractive nonlocal energy $\nlb$ which is also minimized globally by a ball. Precisely, we establish the following as our main result.

\begin{theorem}[Stability of $B_1$] \label{thm:stability}
    Let $d\geq2$, $\alpha\in(0,d-1)$, $\beta>0$. For every $m>0$ there exists $\e(m)\geq0$ such that the unit ball $B_1$ is stable for the functional $\enm$ if and only if $\e\geq\e(m)$. The function $m\mapsto\e(m)$ is continuous, $\e(0)=0$, $\e(m)>0$ for $m\in(0,m_*)$ and $\e(m)\equiv0$ for $m\geq m_*$, where $m_*$ is defined by
    \begin{equation} \label{eq:mstar}
    m_* = m_*(d,\alpha,\beta) \defeq
    \begin{cases}\displaystyle
    \displaystyle \biggl( \frac{\alpha(d-\alpha)(2d+2+\beta)}{\beta(d+\beta)(2d+2-\alpha)}\cdot\frac{\nla(B_1)}{\nlb(B_1)} \biggr)^{\frac{d}{\alpha+\beta}} & \text{if }\beta\geq \beta_*, \\[3ex]
    \displaystyle \biggl( \frac{\alpha(d-\alpha)(2d-\alpha)(d-1+\beta)}{\beta(d+\beta)(2d+\beta)(d-1-\alpha)}\cdot\frac{\nla(B_1)}{\nlb(B_1)} \biggr)^{\frac{d}{\alpha+\beta}}& \text{if }\beta<\beta_*,
    \end{cases}
    \end{equation}
    and 
    \begin{equation} \label{eq:betastar}
    \beta_* = \beta_*(d,\alpha)\defeq \frac{6d+2+\alpha(d-1)}{d-1-\alpha}\,.
    \end{equation}
\end{theorem}

\begin{figure}
	\begin{center}
		\includegraphics[width=6.5cm]{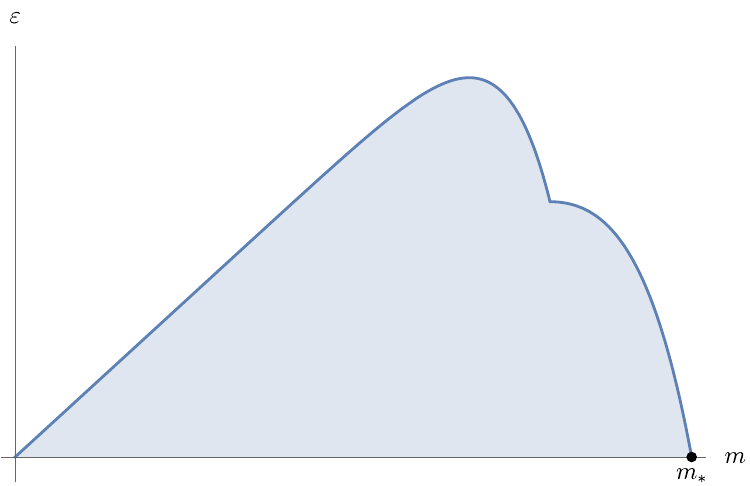}
		\hspace{1cm}
		\includegraphics[width=6.5cm]{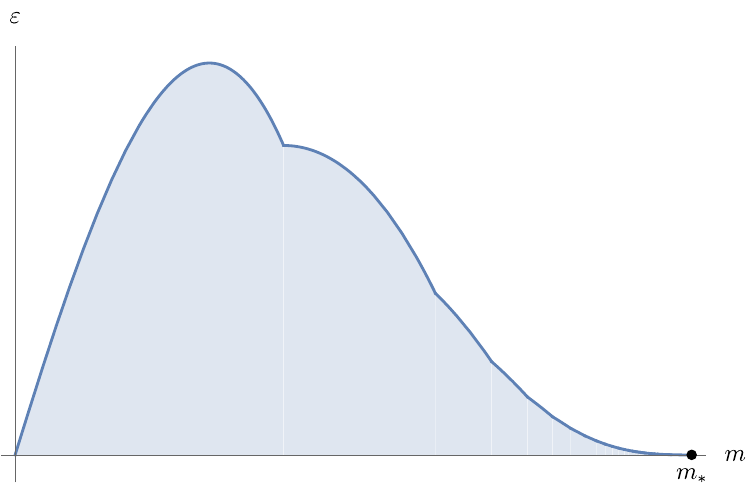}
	\end{center}
	\caption{Numerical plots of the curve $\e(m)$ dividing the stability/instability regions for the unit ball $B_1$ as in Theorem~\ref{thm:stability}; the ball is unstable in the shaded region. Left: the case $\beta>\beta_*$ ($d=3$, $\alpha=1$, $\beta=30$). Right: the case $\beta<\beta_*$ ($d=3$, $\alpha=1$, $\beta=4$). Here $\beta_*=22$.}
	\label{fig:intro}
\end{figure}

See Figure~\ref{fig:intro} for numerical plots of the curve $\e(m)$ in the two cases $\beta>\beta_*$ and $\beta<\beta_*$.
Stability in the previous theorem should be interpreted as the nonnegativity of the second variation of the functional $\enm$ with respect to volume-preserving perturbations of $B_1$ (see Section~\ref{sec:stability}). The proof of Theorem~\ref{thm:stability} follows by combining Proposition~\ref{prop:stab_ball} and Proposition~\ref{prop:eps-reg}, and is based on the analysis of the sign of the quadratic form associated to the second variation, exploiting a spherical harmonics expansion in the spirit of \cite{FFMMM}.

In the case $\beta>\beta_*$ a more precise and explicit description of the function $\e(m)$ is available, see Proposition~\ref{prop:eps-betalarge}. The case $\beta<\beta_*$ is significantly more involved, since the function $\e(m)$ is defined as the supremum of an infinite family of curves; we discuss conjectures about its behavior that are supported by numerical evidence in Section~\ref{sec:numerics}.

By Theorem~\ref{thm:stability} we have in particular that the ball $B_1$ is always stable for all $m\geq m_*$, for all $\e>0$; the constant $m_*$ has indeed been characterized in \cite{BonCriTop24} as the stability threshold of $B_1$ for the functional $\mathcal{F}_{0,m}$, without the perimeter perturbation. On the other hand, denoting by $\e_*\defeq\max_{m\geq0}\e(m)$, we have that $B_1$ is stable for all $\e\geq\e_*$, for all values of $m$. For a fixed $\e<\e_*$, there are instead (at least) two disconnected regions for the parameter $m$ in which the ball is stable, separated by an instability region.

As is common in minimization problems for functionals given by the sum of a perimeter term and a bulk term, the strict stability of $B_1$ is expected to imply its local minimality among sets that are close to the ball in the $L^1$-topology -- that is, sets whose symmetric difference with $B_1$ has sufficiently small volume. This has been established for Gamow's model in \cite{BoCr14,FFMMM}; for the functional $\mathcal{F}_{0}$, due to the lack of the regularizing effect from the perimeter term, we have shown in \cite{BonCriTop24} that the strict stability of $B_1$ ensures its local minimality among sets whose boundaries lie within a small uniform neighbourhood of the boundary of $B_1$. Nevertheless, we do not pursue this question further here.

\medskip

Regarding global minimizers, we prove the existence of solutions of \eqref{eq:min} for all values of $\e$ and $V$ (see Proposition~\ref{prop:existence}), via a compactness result due to Frank and Lieb \cite{FrLi2015} combined with the confining effect of the attraction term. Then we simply collect results from the seminal papers by Figallli, Fusco, Maggi, Morini, and Millot \cite{FFMMM}, and by Frank and Lieb \cite{FraLie21} to obtain the following for the rescaled energy $\enm$:

\begin{theorem}[Global minimality of $B_1$] \label{thm:globalmin}
    Let $d\geq2$, $\alpha\in(0,d-1)$, $\beta>0$. Then there exists a function $\epsglob(m)\geq0$, with $\epsglob(m)\to0$ as $m\to0^+$, such that the unit ball $B_1$ is a global minimizer of $\enm$ if and only if $\e\geq\epsglob(m)$. Moreover, there exists a constant $m_{\rm ball}=m_{\rm ball}(d,\alpha,\beta)>0$ such that $\epsglob(m)>0$ for $m\in(0,m_{\rm ball})$ and $\epsglob(m)\equiv 0$ for $m\geq m_{\rm ball}$.
\end{theorem}

Denoting $\epsglob^* \defeq \max_{m\geq 0} \epsglob(m)$, we see that $B_1$ is the global minimizer of $\enm$ for all $\e\geq \epsglob^*$ and for all values of $m$. If $\e < \epsglob^*$, then there are instead (at least) two disconnected regions for the parameter $m$ where the ball is a global minimizer.

\medskip

In the rest of the paper we will always assume that $\alpha\in(0,d-1)$ and $\beta>0$ are fixed parameters. The paper is organized as follows. In Section~\ref{sec:stability} we study the second variation of the functional $\enm$ at the unit ball $B_1$ and we prove in particular Theorem~\ref{thm:stability}. We discuss global minimizers in Section~\ref{sec:global_min}, which contains the proof of Theorem~\ref{thm:globalmin}. In Section~\ref{sec:numerics} we provide some additional numerical insights on the shape of the stability curve $\e(m)$.




\section{Stability of the ball}\label{sec:stability}

In this section we discuss the volume-constrained stability of the unit ball $B_1$ for the family of functionals $\enm$ introduced in \eqref{eq:energy-scaled}. Given a set $E$ with finite volume and finite perimeter, we say that $E$ is a \emph{volume-constrained stationary set} for $\enm$ if the first variation of the energy at $E$ along any volume-preserving flow induced by any vector field $X \in C^\infty_c(\R^d;\R^d)$ is zero, i.e., if
\begin{equation*}
    \frac{\dd}{\dd t}\enm(\Phi_t(E))|_{t=0}=0,
\end{equation*}
where $\Phi_t$ is such that $\frac{\partial}{\partial t}\Phi_t(x) = X(\Phi_t(x))$, $\Phi_0(x)=x$, and satisfies $|\Phi_t(E)|=|E|$ for all $t$. The set $E$ is a \emph{volume-constrained stable set} for $\enm$ if it is stationary and in addition the second variation at $E$ along a volume-preserving flow induced by any vector field $X$ is nonnegative, i.e., if
\begin{equation*}
    \frac{\dd^2}{\dd t^2}\enm(\Phi_t(E))|_{t=0}\geq0
\end{equation*}
where $\Phi_t$ is as before.


\bigskip

\subsection{Stability of the ball}
The ball $B_1$ is a volume-constrained stationary set for $\enm$ since $\pt B_1$ has constant mean curvature and the potential $\int_{B_1} |x-y|^\sigma\dd y$ is constant for any $x\in\pt B_1$.
In order to study the stability of $B_1$ we introduce the following quadratic form:
\begin{align} \label{eq:quad_form}
&\mathcal{Q} \enm (u) \defeq
\e \left[ \int_{\pt B_1} |\nabla_{\tau} u|^2 \dd\Hd - (d-1) \int_{\pt B_1} u^2 \dd \Hd \right] \nonumber \\
& + m^{(d-\alpha+1)/d}\left[- \int_{\partial B_1}\int_{\partial B_1} \frac{|u(x)-u(y)|^2}{|x-y|^\alpha}\dd\Hd_x\dd\Hd_y + c^2_{-\alpha,\partial B_1}\int_{\partial B_1} u^2\dd\Hd \right] \\
& + m^{(d+\beta+1)/d}\left[-\int_{\partial B_1}\int_{\partial B_1} |x-y|^\beta |u(x)-u(y)|^2\dd\Hd_x\dd\Hd_y +  c^2_{\beta,\partial B_1} \int_{\partial B_1} u^2\dd\Hd\right], \nonumber
\end{align} 
where
\begin{equation} \label{eq:nlcurvature_ball}
c^2_{\sigma,\partial B_1}(x) \defeq \int_{\partial B_1}|x-y|^\sigma |\nu_{B_1}(x)-\nu_{B_1}(y)|^2\dd\Hd_y \qquad\text{for all } x\in\partial B_1,
\end{equation}
the symbol $\nabla_\tau$ denotes the tangential gradient on $\partial B_1$, and $\nu_{B_1}$ the unit outward normal to $\pt B_1$. Notice that $c^2_{\sigma,\partial B_1}(x)$ is constant on $\partial B_1$, and its value can be computed explicitly, see \cite[Remark~2.3]{BonCriTop24}. By \cite[Remark~6.2]{FFMMM} and \cite[Theorem~2.1]{BonCriTop24}, the second variation of $\enm$ at $B_1$ along a volume-preserving flow induced by the vector field $X$ is given by $\mathcal{Q}\enm( X \cdot \nu_{B_1})$. 

The quadratic form $\mathcal{Q}\enm(u)$ can be expressed in terms of the Fourier decomposition of $u$ with respect to the orthonormal basis of spherical harmonics. If $\mathcal{S}_k$ is the finite dimensional subspace of spherical harmonics of degree $k\in\N\cup\{0\}$, and $\{Y_k^i\}_{i=1}^{d(k)}$ is an orthonormal basis (of dimension $d(k)$) for $\mathcal{S}_k$ in $L^2(\partial B_1)$ (see for instance \cite{Gro}), then for $u\in L^2(\partial B_1)$ we let
\begin{equation} \label{eq:fourier}
a^i_k(u)\defeq \int_{\partial B_1} u \, Y^i_k \dd\Hd
\end{equation}
be the Fourier coefficient of $u$ corresponding to $Y^i_k$ so that we have
\begin{equation} \label{eq:L2norm}
\|u\|_{L^2(\partial B_1)}^2 = \sum_{k=0}^\infty\sum_{i=1}^{d(k)} \bigl(a^i_k(u)\bigr)^2.
\end{equation}
Combining the results in \cite{FFMMM} and \cite{Asc22} (see also \cite[Section 2.3]{BonCriTop24}), one easily obtains the following:
\begin{proposition} \label{prop:quad_form}
For $u \in L^2(\pt B_1)$ we have
    \begin{multline} \label{eq:IIvar_spherical_harm}
\mathcal{Q}\enm(u) = \sum_{k=0}^\infty \sum_{i=1}^{d(k)} \Bigl[ \e \bigl(\lambda_k-\lambda_1\bigr) 
+ m^{(d-\alpha+1)/d}\bigl(\mu_1(-\alpha)-\mu_k(-\alpha)\bigr) \\ + m^{(d+\beta+1)/d}\bigl(\mu_1(\beta)-\mu_k(\beta) \bigr)\Bigr]\bigl(a^i_k(u)\bigr)^2.
\end{multline}
Here $\lambda_k=k(k+d-2)$ for $k \geq 1$, $\lambda_0=0$, and for $\sigma\in(-(d-1),\infty)$,
\begin{equation} \label{eq:mu}
\mu_{k}(\sigma) \defeq (d-1)\omega_{d-1}2^{d-1+\sigma} \, \frac{\Gamma\left(\frac{d-1+\sigma}{2}\right)\Gamma\left(\frac{d-1}{2}\right)}{\Gamma\left(\frac{2d-2+\sigma}{2}\right)}\Biggl[1-\prod_{j=0}^{k-1}\frac{j-\frac{\sigma}{2}}{j+d-1+\frac{\sigma}{2}}\Biggr]
\end{equation}
for $k\geq1$, and $\mu_0(\sigma)=0$.
\end{proposition}

\medskip

\begin{remark}[Properties of $(\lambda_k)_k$ and $(\mu_k(\sigma))_k$]\label{rmk:prop_eigenvalues}
Clearly, by definition, the sequence $(\lambda_k)_k$ is strictly increasing. By \cite[Lemma~2.6]{BonCriTop24}, for any $-(d-1)<\sigma<0$ the sequence $(\mu_k(\sigma))_k$ is strictly increasing and $\mu_k(\sigma) - \mu_1(\sigma) \geq \mu_2(\sigma)-\mu_1(\sigma) = -\frac{\sigma \mu_1(\sigma)}{2d+\sigma}>0$ for all $k\geq2$.
If $\sigma>0$, on the other hand, $\max_{k\geq1}\mu_{k}(\sigma)=\mu_1(\sigma)$ and $\mu_1(\sigma) - \mu_k(\sigma) \geq C_{d,\sigma}>0$ for all $k\geq2$, for some constant $C_{d,\sigma}$ depending only on $d$ and $\sigma$. Furthermore, the sequence $(\mu_k(\sigma))_k$ is bounded, with $\lim_{k\to\infty}\mu_k(\sigma)$ given by an explicit constant. Finally, one has $\mu_1(\sigma)= c^2_{\sigma,\partial B_1}$ where $c^2_{\sigma,\partial B_1}$ is the constant in \eqref{eq:nlcurvature_ball}.
\end{remark}

\medskip

For any $k \geq 2$, we define the functions $\e_k \colon [0,\infty) \to \R$ by
\begin{equation}\label{eq:eps_k}
        \e_k(m) \defeq \frac{m^{(d-\alpha+1)/d}\Bigl[ \bigl(\mu_k(-\alpha)-\mu_1(-\alpha)\bigr) - m^{(\alpha+\beta)/d}\bigl(\mu_1(\beta)-\mu_k(\beta)\bigr) \Bigr]}{\lambda_k-\lambda_1}\,,
\end{equation}
and, for notational convenience, we also set $\e_1(m)\equiv0$. Let for $m\geq0$
\begin{equation} \label{eq:eps_func}
        \e(m) \defeq \sup_{k \geq 1} \e_k(m).
\end{equation}

\begin{proposition}\label{prop:stab_ball}
    Let $m>0$, $\alpha\in(0,d-1)$ and $\beta>0$. The ball of unit volume $B_1$ is a volume-constrained stable set for $\enm$ if and only if $\e \geq \e(m)$ with $\e(m)$ defined in \eqref{eq:eps_func}.
\end{proposition}

\begin{proof}
    Arguing as in \cite[Theorem~2.7]{BonCriTop24} and \cite[Theorem~7.1]{FFMMM}, we have that $B_1$ is a volume-constrained stable set for $\enm$ if and only if $\mathcal{Q}\enm(u) \geq 0$ for every $u\in C^\infty(\pt B_1)$ with $\int_{\pt B_1} u \dd \Hd = 0$.

    Note that since $\int_{\pt B_1} u \dd \Hd =0$, we have $a_0^1(u)=0$. Hence, the sum in \eqref{eq:IIvar_spherical_harm} starts at $k=2$ (since each difference also vanishes when $k=1$) for every $u\in C^\infty(\pt B_1)$ with $\int_{\pt B_1} u \dd \Hd=0$. Then simply $\mathcal{Q}\enm(u) \geq 0$ for any $u\in C^\infty(\pt B_1)$ with $\int_{\pt B_1} u \dd \Hd=0$ if and only if $\e \geq \e(m)$.   
\end{proof}

\begin{remark} \label{rmk:mstar}
In \cite[Theorem~1.1]{BonCriTop24} we show that the ball $B[V]$ of volume $V$ is a volume-constrained stable set for the functional $\nla+\nlb$ if and only if $V/\omega_d \geq m_*$, where $m_*$ is the explicit constant defined in \eqref{eq:mstar} depending only on $d$, $\alpha$, and $\beta$
\end{remark}


\bigskip

\subsection{Properties of $\e_k(m)$}
Next we would like to investigate the properties of the functions $\e_k(m)$ defined in \eqref{eq:eps_k}, and in turn of their supremum $\e(m)$. To this end we introduce the following notation. We rewrite the function $\e_k(m)$, for $k\geq2$, as
    \[
        \e_k(m) = m^r \bigl(A_k - B_k m^q\bigr),
    \]
where 
    \begin{gather}
        r = \frac{d-\alpha+1}{d}, \qquad q=\frac{\alpha+\beta}{d}, \label{eq:r_and_q} \\
        A_k = \frac{\mu_k(-\alpha)-\mu_1(-\alpha)}{\lambda_k-\lambda_1}, \qquad B_k=\frac{\mu_1(\beta)-\mu_k(\beta)}{\lambda_k-\lambda_1}. \label{eq:A_k_and_B_k}
    \end{gather}
Notice that, in view of Remark~\ref{rmk:prop_eigenvalues}, the quantities $A_k$ and $B_k$ are strictly positive for all $k\geq2$. As a consequence of the simple form of $\e_k(m)$, we obtain some basic properties of $\e_k$ which we list in the next lemma, see also Figure~\ref{fig:plots_1} (left).

\begin{lemma} \label{lem:basic_prop_eps_k}
    For any $k\geq 2$, $\alpha\in(0,d-1)$ and $\beta>0$, we have the following:
        \begin{enumerate}
            \item $\e_k(m)$ has two roots: $m=0$ and $m_k^0 \defeq \bigl(\frac{A_k}{B_k}\bigr)^{1/q}$.
            \item $\e_k(m)$ has one nonzero critical point: $m_k^c \defeq \left( \frac{rA_k}{(r+q)B_k}\right)^{1/q}$.
            \item $\e_k(m)$ is increasing for $m<m_k^c$, decreasing for $m>m_k^c$, and it achieves its maximum value at $m_k^c$:
                \[
                    \max_{m\in [0,\infty)} \e_k(m) = \e_k(m_k^c)
                    = \left(\frac{r A_k}{(r+q)B_k}\right)^{r/q}\left(\frac{q A_k}{r+q}\right).
                \]
            \item $\e_k(m)$ is strictly concave on $(0,\infty)$ if $\alpha>1$, whereas, if $\alpha<1$, $\e_k(m)$ is convex for $m<(\frac{1-\alpha}{1+\beta})^{1/q}m_k^c$ and concave otherwise.
        \end{enumerate}
\end{lemma}

\begin{figure}
	\begin{center}
		\includegraphics[width=6.5cm]{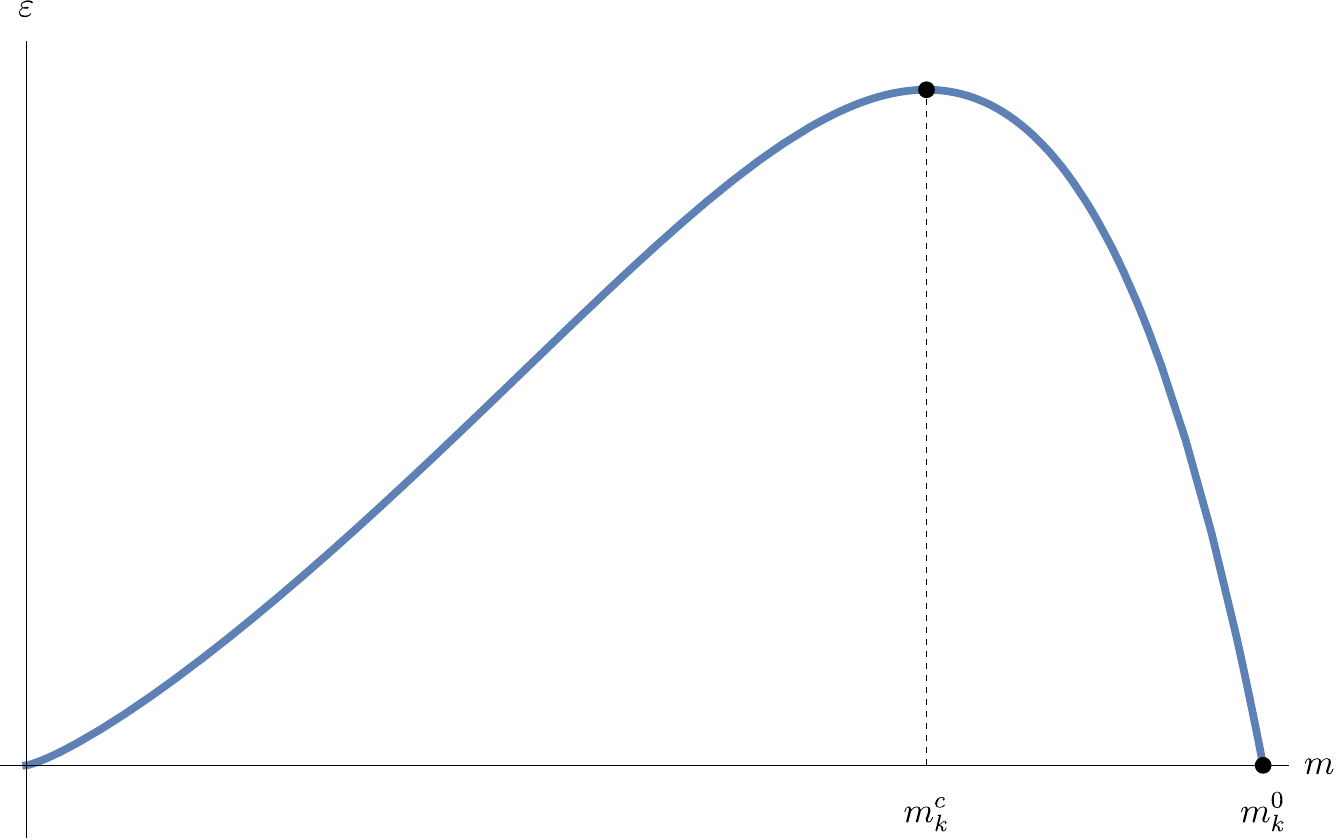}
		\hspace{1cm}
		\includegraphics[width=6.5cm]{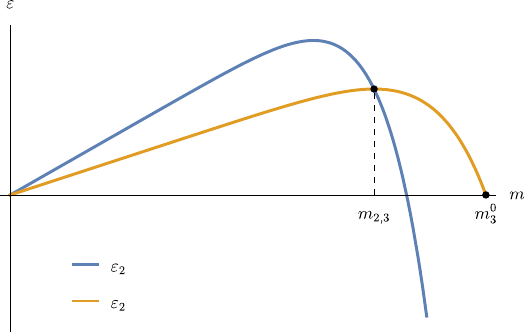}
	\end{center}
	\caption{Left: numerical plot of the shape of one of the functions $\e_k$ as described in Lemma~\ref{lem:basic_prop_eps_k} (the function plotted is $\e_5(m)$ for the values of the parameters $d=3$, $\alpha=0.2$, $\beta=15$). Right: the intersection of $\e_2$ and $\e_3$ is at the maximum point of $\e_3$, see Lemma~\ref{lemma:eps_2eps_3} (here $d=3$, $\alpha=1$, $\beta=24$).}
	\label{fig:plots_1}
\end{figure}

\begin{remark} \label{rmk:mstar2}
    By the results in \cite{BonCriTop24}, see in particular the equations (2.19) and (2.21) in \cite[Theorem~2.7]{BonCriTop24}, the constant $m_*$ defined in \eqref{eq:mstar} coincides with the supremum of the positive roots of the function $\e_k$:
    \begin{equation} \label{eq:mstar2}
        m_*=\sup_{k \geq 2} m_k^0.
    \end{equation}
    Furthermore, the supremum in \eqref{eq:mstar2} is attained in the limit as $k\to\infty$ if $\beta<\beta_*$, and is attained for $k=3$ in the case $\beta\geq\beta_*$, where $\beta_*$ is the threshold defined in \eqref{eq:betastar} depending on $d$ and $\alpha$. Finally, by \cite[Lemma~A.1]{BonCriTop24} we also have
    \begin{equation} \label{eq:inf_mk}
        m_2^0 = \inf_{k\geq 2}m_k^0.
    \end{equation}
\end{remark}

Another important value on $[0,\infty)$ for our discussion is the intersection point of two curves $\e_k$ and $\e_\ell$. Clearly $\e_k$ and $\e_\ell$ intersect when $m=0$ for any $k\neq \ell$. If any two such curves intersect at a nonzero point, then this point is given by
\begin{equation}\label{eq:m_kl}
        m_{k,\ell} \defeq \left(\frac{A_k-A_\ell}{B_k - B_\ell}\right)^{1/q}.
\end{equation}
We will show in Lemma~\ref{lemma:A_k} below that the sequence $(A_k)_{k\geq2}$ is strictly decreasing. Hence the existence of the intersection point $m_{k,\ell}$ depends on the sign of $B_k-B_\ell$, which might however change along the sequence.

In the following lemma we show that the curve $\e_3$ has its maximum exactly at the intersection point $m_{2,3}$ with the curve $\e_2$, see also Figure~\ref{fig:plots_1} (right).

\begin{lemma} \label{lemma:eps_2eps_3}
    It holds $m_{2,3}=m_3^c$.
\end{lemma}

\begin{proof}
    We first introduce a convenient notation for the coefficients $A_k$ and $B_k$, which will also be useful later. We define
    \begin{equation} \label{eq:coeff_1}
        \ell \defeq \frac{d-1}{2}, \quad t\defeq \ell - \frac{\alpha}{2}, \quad \tau \defeq \ell + \frac{\beta}{2},
    \end{equation}
    and for $k\geq 2$
    \begin{equation} \label{eq:abcd}
        a_k \defeq \prod_{j=1}^{k-1}(j+l-t), \quad
        b_k \defeq \prod_{j=1}^{k-1}(j+l+t), \quad
        c_k \defeq \prod_{j=1}^{k-1}(j+l-\tau), \quad
        d_k \defeq \prod_{j=1}^{k-1}(j+l+\tau).
    \end{equation}
    Then, in view of \eqref{eq:A_k_and_B_k} and \eqref{eq:mu}, it is straightforward to deduce the following expressions for all $k\geq 2$:
    \begin{equation} \label{eq:AB}
        A_k = \kappa_{d,\alpha}\frac{1-\frac{a_k}{b_k}}{(k-1)(k+d-1)}\,, \qquad
        B_k = \kappa_{d,\beta}\frac{1-\frac{c_k}{d_k}}{(k-1)(k+d-1)}\,,
    \end{equation}
    for suitable positive constants $\kappa_{d,\sigma}$ depending only on $d$ and $\sigma$.

    We can now prove the identity in the statement. We first compute using \eqref{eq:abcd}
    \begin{align*}
        b_3 - a_3 = (d+2)(d-1-\alpha), \qquad d_3-c_3 = (d+2)(d-1+\beta).
    \end{align*}
    Then by \eqref{eq:AB}
    \begin{align} \label{eq:lemmaeps_2eps_3-1}
        \frac{A_3-A_2}{A_3}
        & = 1 - \frac{A_2}{A_3}
        = 1 - \frac{2(d+2)}{(d+1)}\cdot\frac{(b_2-a_2)}{(b_3-a_3)}\cdot\frac{b_3}{b_2} \nonumber\\
        & = 1 - \frac{2(d+2)}{(d+1)}\cdot\frac{(d-1-\alpha)}{(d+2)(d-1-\alpha)}\cdot(d+1-{\textstyle\frac{\alpha}{2}})
        = - \frac{d+1-\alpha}{d+1}\,,
    \end{align}
    and by an analogous computation
    \begin{align} \label{eq:lemmaeps_2eps_3-2}
        \frac{B_3-B_2}{B_3} = - \frac{d+1+\beta}{d+1}\,.
    \end{align}
    Notice in particular that, from these expressions, $A_2>A_3$ and $B_2>B_3$. Hence by \eqref{eq:m_kl} the two curves $\e_2(m)$ and $\e_3(m)$ always intersect at the point $m_{2,3}>0$. Now, by comparing the expression of $m_{2,3}$ in \eqref{eq:m_kl} and that of $m_3^c$ in Lemma~\ref{lem:basic_prop_eps_k}, we have that $m_{2,3}=m_3^c$ if and only if
    \begin{equation*}
        \frac{A_3-A_2}{B_3-B_2} = \frac{(d+1-\alpha)}{(d+1+\beta)}\cdot\frac{A_3}{B_3}\,,
    \end{equation*}
    and it is easily seen that this identity holds thanks to \eqref{eq:lemmaeps_2eps_3-1} and \eqref{eq:lemmaeps_2eps_3-2}.
\end{proof}

In the following technical preparatory lemma, we prove the strict monotonicity of the sequence $(A_k)_{k\geq2}$. Since $\e_k(m)=A_k m^r + o(m^r)$ as $m\to0^+$, where $\lim_{m\to0^+}\frac{o(m^r)}{m^r}=0$, the coefficient $A_k$ determines the behaviour of the curve $\e_k(\cdot)$ at the origin. In particular, as a consequence of the lemma, we have that $\e_k(m)>\e_{k+1}(m)$ for $m$ in a right neighbourhood of the origin, for all $k\geq2$.

\begin{lemma} \label{lemma:A_k}
    The sequence $(A_k)_{k\geq2}$ is strictly decreasing.
\end{lemma}

\begin{proof}
    Notice that the inequality $A_2>A_3$ was already established in the proof of Lemma~\ref{lemma:eps_2eps_3}, see in particular \eqref{eq:lemmaeps_2eps_3-1}. In view of the expression \eqref{eq:AB} of $A_k$, the inequality $A_k>A_{k+1}$ is equivalent to
    \begin{equation}\label{eq:lemAk-1}
        \frac{1-\frac{a_{k+1}}{b_{k+1}}}{k(k+d)} < \frac{1-\frac{a_k}{b_k}}{(k-1)(k+d-1)},
    \end{equation}
    where the sequences $(a_k)_{k\geq2}$ and $(b_k)_{k\geq2}$ are defined in \eqref{eq:abcd}. After some algebraic manipulations, we see that \eqref{eq:lemAk-1} is equivalent to
    \begin{equation*}
        (2k+d-1)b_k b_{k+1} - k(k+d)a_k b_{k+1} + (k-1)(k+d-1)b_k a_{k+1}>0,
    \end{equation*}
    and in turn, dividing by $a_k b_{k+1}$, to
    \begin{equation*}
        (2k+d-1)\frac{b_k}{a_k} - k(k+d) + (k-1)(k+d-1)\frac{(k+\ell-t)}{(k+\ell+t)}>0,
    \end{equation*}
    and, finally, to 
    \begin{equation*}
        \frac{(2k+d-1)}{k-1}\biggl(\frac{b_k}{a_k}-1\biggr) - \frac{2t(k+d-1)}{(k+\ell+t)}>0.
    \end{equation*}
    We therefore define the sequences
    \begin{equation*}
        \eta_k \defeq \frac{(2k+d-1)}{k-1}\biggl(\frac{b_k}{a_k}-1\biggr) \qquad\text{and}\qquad \sigma_k \defeq \frac{2t(k+d-1)}{(k+\ell+t)},
    \end{equation*}
    so that the conclusion $A_k>A_{k+1}$ amounts to show the inequality $\eta_k-\sigma_k>0$ for all $k\geq2$.
    
    Notice that $\eta_2-\sigma_2>0$, as this inequality is equivalent to $A_2>A_3$ which has already been established, as observed at the beginning of the proof. Moreover, it is easily seen that $(\sigma_k)_{k\geq2}$ is a strictly decreasing sequence. Therefore, if we show that $\eta_k\geq\eta_2$ for all $k\geq3$, it would follow that $\eta_k-\sigma_k > \eta_2 - \sigma_2 >0$ for all $k\geq3$, which is the desired conclusion.

    \medskip
    To conclude the proof, it only remains to show the inequality $\eta_k\geq\eta_2$ for all $k\geq2$. We proceed by induction. After some algebraic manipulations, the inequality $\eta_k\geq\eta_2$ is equivalent to
    \begin{equation}
        \frac{b_k}{a_k} \geq 1 + \frac{2t(d+3)(k-1)}{(2k+d-1)(1+\ell-t)}\,. \tag*{$(\square)_{k}$} \label{eq:lemAk-induction}
    \end{equation}
    Assume now that \ref{eq:lemAk-induction} is true for some $k\geq2$, and let us prove $(\square)_{k+1}$. We have
    \begin{align*}
        \frac{b_{k+1}}{a_{k+1}} - \Biggl[ & 1 + \frac{2t(d+3)k}{(2k+d+1)(1+\ell-t)} \Biggr]
        = \frac{(k+\ell+t)}{(k+\ell-t)}\frac{b_k}{a_k} - \Biggl[ 1 + \frac{2t(d+3)k}{(2k+d+1)(1+\ell-t)} \Biggr] \\
        & \geq \frac{(k+\ell+t)}{(k+\ell-t)}\Biggl[ 1 + \frac{2t(d+3)(k-1)}{(2k+d-1)(1+\ell-t)} \Biggr] - \Biggl[ 1 + \frac{2t(d+3)k}{(2k+d+1)(1+\ell-t)} \Biggr] \\
        & = \frac{2t}{k+\ell-t} + \frac{(k+\ell+t)}{(k+\ell-t)}\cdot\frac{2t(d+3)(k-1)}{(2k+d-1)(1+\ell-t)} - \frac{2t(d+3)k}{(2k+d+1)(1+\ell-t)}\,,
    \end{align*}
    where we used the induction assumption \ref{eq:lemAk-induction} in the second passage. The goal is to show that the previous quantity is nonnegative, that is, after converting to a (positive) common denominator, to show that
    \begin{multline} \label{eq:lemAk-3}
        Q_{d,\alpha}(k)\defeq (1+\ell-t)(2k+d-1)(2k+d+1)+(2k+d+1)(k+\ell+t)(d+3)(k-1)\\-(2k+d-1)(k+\ell-t)(d+3)k \geq0.
    \end{multline}
    Again by elementary algebraic manipulations we find
    \begin{align*}
        Q_{d,\alpha}(k)
        & = (1+\ell-t)(2k+d-1)(2k+d+1)+(d+3)\Bigl[2kt(2k+d-1)-(d+1)(k+\ell+t)\Bigr] \\
        & = A_{d,\alpha}k^2 + B_{d,\alpha} k + C_{d,\alpha},
    \end{align*}
    with
    \begin{gather*}
        A_{d,\alpha} \defeq 4(1+\ell+2t+td)>0, \\
        B_{d,\alpha} \defeq 4d(1+\ell-t)+(d+3)\bigl[2t(d-1)-d-1\bigr].
    \end{gather*}
    Hence the sequence of points $(Q_{d,\alpha}(k))_{k\geq2}$ are on a parabola whose vertex has $x$-coordinate $v=-\frac{B_{d,\alpha}}{2A_{d,\alpha}}$. We can compute
    \begin{align*}
        v-2 & = \frac{1}{2A_{d,\alpha}}\bigl(-B_{d,\alpha}-4A_{d,\alpha}\bigr) \\
        & = \frac{1}{2A_{d,\alpha}}\Bigl[ -4d(1+\ell) + (d+3)(d+1) -16(1+\ell) -2t(d^2+8d+13)\Bigr] \\
        & \leq \frac{1}{2A_{d,\alpha}}\Bigl[ -4d(1+\ell) + (d+3)(d+1) -16(1+\ell) \Bigr]
        = -\frac{(d+5)(d+1)}{2A_{d,\alpha}}<0,
    \end{align*}
    where we used the fact that $t=\ell-\frac{\alpha}{2}\geq0$ and $A_{d,\alpha}>0$ in the third passage.
    
    Hence the $x$-coordinate $v$ of the vertex of the parabola is smaller than 2, so that the sequence $(Q_{d,\alpha}(k))_{k\geq2}$ is increasing for $k\geq2$ and in particular $Q_{d,\alpha}(k)\geq Q_{d,\alpha}(2)$ for all $k\geq2$. We can explicitly compute, using the definition \eqref{eq:coeff_1} of $\ell$ and $t$ and the assumption $\alpha<d-1$,
    \begin{align*}
        Q_{d,\alpha}(2) &= (d+3)\bigl[ d^2 + 3d - 2 -\alpha(d+3) \bigr] \\
        & \geq (d+3)\bigl[ d^2 + 3d - 2 -(d-1)(d+3) \bigr] = (d+3)(d+1)>0.
    \end{align*}
    We conclude that $Q_{d,\alpha}(k)>0$ for all $k\geq2$. This proves \eqref{eq:lemAk-3} and, in turn, $(\square)_{k+1}$, concluding the proof of the lemma.
\end{proof}


\bigskip

\subsection{Properties of $\e(m)$}
We now qualitatively describe the shape of the function $\e(m)$, defined in \eqref{eq:eps_func}, which according to Proposition~\ref{prop:stab_ball} separates the stability/instability regions of the unit ball $B_1$ for the functional $\enm$. In particular, in this subsection we complete the proof of Theorem~\ref{thm:stability}. We start by proving some general properties of the function $\e(m)$.

\begin{proposition} \label{prop:eps-reg}
Let $m_*>0$ be defined by \eqref{eq:mstar}, depending on $d\geq2$, $\alpha\in(0,d-1)$, $\beta>0$. Then $\e(0)=0$, $\e(m)>0$ for $m\in(0,m_*)$, and $\e(m)=0$ for $m\geq m_*$.

For each $m\in(0,m_*)$, the supremum in \eqref{eq:eps_func} is attained at some finite $\bar{k}(m)$. The function $\e(m)$ is globally continuous and is locally Lipschitz continuous in $(0,m_*)$. Moreover one has the bound $\e(m)\leq A_2 m^{(d-\alpha+1)/d}$, where $A_2$ is as in \eqref{eq:AB}.
\end{proposition}

\begin{proof}
Obviously $\e(0)=0$, as $\e_k(0)=0$ for all $k$. Recalling the identity \eqref{eq:mstar2}, it is immediate to deduce that, if $m\in(0,m_*)$, we have $m<m_{\tilde{k}}^0$ for some $\tilde{k}\geq2$ and hence $\e(m)\geq\e_{\tilde{k}}(m)>0$. On the contrary, if $m\geq m_*$ then all the functions $\e_k$ vanish at $m$, and therefore $\e(m)=0$.

Next, notice that $A_k\to0$ and $B_k\to0$ as $k\to\infty$, as a consequence of the definition \eqref{eq:A_k_and_B_k} of $A_k$ and $B_k$, of the fact that $\lambda_k\to\infty$, and of the boundedness of the sequences $(\mu_k(-\alpha))_k$ and $(\mu_k(\beta))_k$, see Remark~\ref{rmk:prop_eigenvalues}. It follows that $\lim_{k\to\infty}\e_k(m)=0$ for any fixed $m>0$, and this convergence is uniform in $[0,m_*]$.

Let us now prove that $\e(m)$ is Lipschitz continuous in each interval $[m_1,m_2]\subset(0,m_*)$. Again by \eqref{eq:mstar2}, there exists $\tilde{k}\geq2$ such that $c\defeq\inf_{m\in[m_1,m_2]}\e_{\tilde{k}}(m)>0$. Since $\e_k(m)\to0$ uniformly in $[m_1,m_2]$ as $k\to\infty$, we have $\sup_{m\in[m_1,m_2]}\e_k(m)<c$ for all $k$ sufficiently large, that is, in the interval $[m_1,m_2]$ the supremum \eqref{eq:eps_func} defining $\e(m)$ is actually a maximum of finitely many smooth functions. Therefore $\e(m)$ is Lipschitz in $[m_1,m_2]$, as claimed.

It only remains to show the continuity of $\e(m)$ at $0$ and at $m_*$. Fix $\eta>0$. Then, by the uniform convergence of $\e_k(m)\to0$ as $k\to\infty$, we have $\sup_m\e_k(m)<\eta$ for all $k\geq k_\eta$, for some $k_\eta\in\N$. Moreover, by \eqref{eq:mstar2}, for every $k\geq2$ there exists $\delta_k>0$ such that $\e_k(m)<\eta$ for all $m\in[m_*-\delta_k,m_*]$. Then, taking $\delta=\min_{2\leq k \leq k_\eta}\delta_k>0$, one has $\e_k(m)\leq \eta$ for all $m\in[m_*-\delta,m_*]$ and for all $k\geq2$, and hence $0\leq \e(m)\leq\eta$ for all $m\in[m_*-\delta,m_*]$. This shows the continuity from the left of $\e(m)$ at $m_*$.

The continuity at $0$ follows by an analogous argument, using the convergence $\e_k(m)\to0$ as $m\to0$, and the uniform convergence $\e_k(m)\to0$ as $k\to\infty$.

Finally, the bound $\e(m)\leq A_2 m^{(d-\alpha+1)/d}$ is immediate by \eqref{eq:eps_k}, \eqref{eq:A_k_and_B_k}, and Lemma~\ref{lemma:A_k}.
\end{proof}

We can more accurately describe the profile $\e(m)$ when $\beta\geq \beta_*$, where $\beta_*$ is the threshold defined in \eqref{eq:betastar} depending on $d$ and $\alpha$. Indeed, it turns out that in this case the supremum defining $\e(m)$ in \eqref{eq:eps_func} is attained either at $k=2$ or at $k=3$. See Figure~\ref{fig:betalarge} for an illustration.

\begin{figure}
	\begin{center}
		\includegraphics[width=8cm]{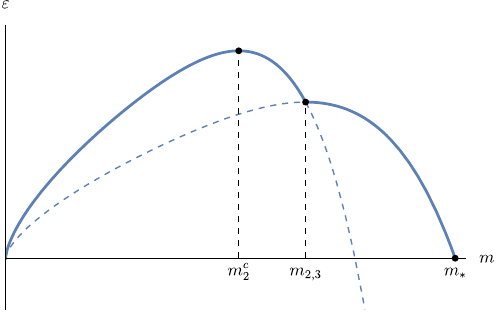}
	\end{center}
	\caption{The function $\e(m)$ in the case $\beta\geq\beta_*$ is just the maximum between the first two curves $\e_2(m)$ and $\e_3(m)$, see Proposition~\ref{prop:eps-betalarge} (the plot is obtained for $d=6$, $\alpha=3$, $\beta=30$).}
	\label{fig:betalarge}
\end{figure}

\begin{proposition}[Shape of $\e(m)$ in the case $\beta\geq\beta_*$] \label{prop:eps-betalarge}
Assume that $\beta\geq\beta_*$, where $\beta_*$ is defined in \eqref{eq:betastar}. Then
\begin{equation} \label{eq:eps-betalarge}
    \e(m) =
    \begin{cases}
        \e_2(m) & \text{if }m\in[0,m_{2,3}],\\
        \e_3(m) & \text{if }m\in[m_{2,3},m_*],\\
        0 & \text{if }m\geq m_*,
    \end{cases}
\end{equation}
where $m_*$ is defined in \eqref{eq:mstar}, and $m_{2,3}$ is the intersection point of $\e_2$ and $\e_3$, see \eqref{eq:m_kl}. The function $\e(m)$ is increasing for $m\in[0,m_2^c]$ and decreasing for $m\in[m_2^c,m_*]$, where $m_2^c$ is the maximum point of the curve $\e_2$ (see Lemma~\ref{lem:basic_prop_eps_k}). The previous points are ordered as $0<m_2^c<m_{2,3}<m_*$.
\end{proposition}

\begin{proof}
First notice, as recalled in Remark~\ref{rmk:mstar2}, that by the results of \cite{BonCriTop24} in the case $\beta\geq\beta_*$ we have $m_*=m_3^0$, where $m_3^0$ is the positive root of the function $\e_3(m)$. Moreover, by \eqref{eq:mstar2} $m_k^0\leq m_*$ for all $k\geq2$.

We first claim that
\begin{equation} \label{proof-betalarge-1}
    \e_k(m)\leq \e_3(m) \qquad\text{for all }m\in[0,m_*], \text{ for all }k>3.
\end{equation}
Indeed, by Lemma~\ref{lemma:A_k} we have $A_k<A_3$ for all $k>3$. Since $\e_k(m)=A_k m^r + o(m^r)$ as $m\to0^+$, where $\lim_{m\to0^+}\frac{o(m^r)}{m^r}=0$, it follows that the inequality \eqref{proof-betalarge-1} holds for all $m$ in a sufficiently small neighbourhood of the origin. The two curves $\e_3$ and $\e_k$ intersect at most once beyond the origin, at the point $m_{3,k}$ (see \eqref{eq:m_kl}). If by contradiction $m_{3,k}<m_*$, which would violate the inequality \eqref{proof-betalarge-1}, then necessarily $m_k^0>m_3^0$ (or else the two curves would intersect a second time beyond the origin). However, since $m_3^0=m_*$ as observed at the beginning of the proof, this contradicts the inequality $m_k^0\leq m_*$. The claim \eqref{proof-betalarge-1} follows.

As a consequence of $\eqref{proof-betalarge-1}$, we have
\begin{equation} \label{proof-betalarge-2}
    \e(m) = \max\bigl\{\e_2(m),\e_3(m)\bigr\} \qquad\text{for all }m\in[0,m_*].
\end{equation}
By Lemma~\ref{lemma:eps_2eps_3}, the two functions $\e_2$ and $\e_3$ intersect at the point $m_{2,3}=m_3^c$ where $\e_3$ has its maximum. In particular $\e_3(m_{2,3})>0$ and therefore $m_{2,3}<m_*$. Since $A_2>A_3$ by Lemma~\ref{lemma:A_k}, we have as in the previous part of the proof that close to the origin $\e_2(m)>\e_3(m)$, and since the two curves intersect only at $m_{2,3}$ we obtain \eqref{eq:eps-betalarge}.

Clearly $m_2^c<m_{2,3}$, and $\e(m)$ is increasing in $[0,m_2^c]$ and decreasing in $[m_2^c,m_{2,3}]$, as it coincides with $\e_2(m)$ in these two intervals. Since $m_{2,3}$ coincides with the maximum point of $\e_3(m)$, we have that $\e_3(m)$ (that is, $\e(m)$) is decreasing in $[m_{2,3},m_*]$.
\end{proof}

\medskip

\begin{remark} The picture in the case $0<\beta<\beta_*$ is more involved, due to the fact that one has to consider \emph{all} curves $\e_k(m)$ when finding their supremum and not simply the first two. We discuss this issue more in detail, and provide some numerical insights and conjectures in Section~\ref{sec:numerics}.
\end{remark}


\section{Global minimality}\label{sec:global_min}


Before we characterize the parameter regimes where the unit ball is the global minimizer, we prove that a minimizer of $\enm$ exists \emph{for any} $\e,m>0$. This improves the existence result in \cite[Lemma~6.1]{Asc22} where the author obtains the existence of a minimizer only for large values of $m$. For notational simplification, we state and prove the existence results for the unscaled functional $\en$.

\begin{proposition}\label{prop:existence}
    Let $d\geq 2$, $\alpha\in(0,d)$, $\beta>0$. Then for any $\e>0$ and $V>0$ the problem \eqref{eq:min} admits a minimizer.
\end{proposition}

\begin{proof}
    Let $(E_k)_k$ be a minimizing sequence for the problem \eqref{eq:min}. Then $E_k$ have uniformly bounded perimeter. Since $|E_k|=V$ for all $k\geq 1$, by \cite[Proposition~2.1]{FrLi2015}, there exists a set $E_0$ with positive measure $V_0$ and a sequence $(a_k)_k$ in $\R^d$ such that for a subsequence we have that $E_{n_k}-a_k \to E_0$ locally, i.e. $\chi_{(E_{n_k}-a_k)} \to \chi_{E_0}$ in $L^1_{\loc}(\R^d)$. Moreover, 
        \[
            V_0 = |E_0| \leq \liminf_{k\to\infty} |E_{n_k}| = V.
        \]

    Now, denoting $E_{n_k}-a_k$ again by $E_k$ with an abuse of notation, we assume for a contradiction that $V_0<V$. Fix $R_0>0$ such that $|E_0 \cap B_{R_0}| \geq V_0/2$, and let $\eta \defeq V-V_0$. Then for any $R>R_0$ there exists $K_R\in\N$ such that for any $k>K_R$, $|E_k\setminus B_R|\geq \eta/2$. Indeed, by the local convergence of $E_k$ to $E_0$,
        \[
            V_0 \geq |E_0 \cap B_R| = \lim_{k \to \infty} |E_k \cap B_R| = V - \lim_{k\to\infty} |E_k \setminus B_R|;
        \]
    hence, $\lim_{k\to \infty}|E_k \setminus B_R| \geq V-V_0=\eta$.

    Using the uniform bound on the energy and the positivity of the perimeter and the repulsive terms, for $k \geq K_R$, we obtain
        \begin{align*}
            C &\geq \en(E_k) \geq \int_{E_k} \int_{E_k} |x-y|^\beta \dd x \dd y \geq \int_{E_k \cap B_{R_0}}\int_{E_k \setminus B_R} |x-y|^\beta \dd x \dd y \\
                &\geq |R-R_0|^\beta \frac{V_0 \eta}{8}.
        \end{align*}
    Sending $R\to\infty$ this estimate yields the desired contradiction. Therefore $V_0=V$. This implies that no mass is lost in the limit $k \to \infty$. Hence, by the lower semicontinuity of the energy $\en$, $E_0$ is a minimizer.
\end{proof}

\medskip

\begin{proof}[Proof of Theorem~\ref{thm:globalmin}]
Let $d\geq 2$, $\alpha\in(0,d-1)$, and $\beta>0$. By \cite[Theorem~1.3]{FFMMM}, for any $m>0$ there exists $\ol{\e}(m)>0$ (also depending on $d$ and $\alpha$) such that if $\e > \ol{\e}(m)$, then the energy $\e\,\per(E) + m^{(d-\alpha+1)/d}\,\nla(E)$ is minimized uniquely (up to translations) by the unit ball $B_1$ among all sets of volume $\omega_d$. Since $\nlb(E)$ is an attractive energy, it is clear (for example via Riesz rearrangement inequality) that its global minimizer is also the ball. Therefore, among all measurable sets of volume $\omega_d$, the unique global minimizer (up to translations) of $\enm$ is given by the unit ball for any $m>0$ and $\e > \ol{\e}(m)$.

We define
$$
\epsglob(m) \defeq \inf \left\{ \e>0 \,\colon\, B_1 \text{ is a global minimizer of }\mathcal{F}_{\e',m} \text{ for all }\e'>\e \right\}.
$$
Clearly $\epsglob(m)\leq\ol{\e}(m)$ and $B_1$ is a global minimizer of $\enm$ for all $\e\geq\epsglob$. It is easy to see that, if $B_1$ globally minimizes $\enm$ for some $\e>0$, then it also minimizes $\mathcal{F}_{\e',m}$ for all $\e'>\e$. Therefore $B_1$ is not a global minimizer of $\enm$ if $\e<\epsglob$.

Now, let $V_{\rm ball}$ be the volume threshold in \cite[Theorem~1.1]{FraLie21} that depends only on $d$, $\alpha$, and $\beta$. Then, taking $m_{\rm ball} \defeq \left(\frac{V_{\rm ball}}{\omega_d}\right)$, we obtain that when $m>m_{\rm ball}$, the energy $\mathcal{F}_{0,m}=m^{(d-\alpha+1)/d}\, \nla(E)+m^{(d+\beta+1)/d}\,\nlb(E)$ is uniquely minimized, up to translations, among all measurable sets of volume $\omega_d$, by the unit ball $B_1$. Since the perimeter term is also minimized by the ball, again, the global minimizer of \eqref{eq:energy-scaled} is given by the unit ball for any $\e>0$ and for any $m>m_{\rm ball}$. In particular, $\epsglob(m)=0$ for all $m\geq m_{\rm ball}$.

We eventually show that $\epsglob(m)>0$ for $m\in(0,m_{\rm ball})$. Indeed, notice that if $B_1$ globally minimizes the functional $\mathcal{F}_{0,m}$ for some $m>0$, then it also minimizes $\mathcal{F}_{0,m'}$ for $m'>m$. It follows that the ball is not a global minimizer of $\mathcal{F}_{0,m}$ when $m\in(0,m_{\rm ball})$. Then necessarily $\epsglob(m)>0$ for such $m$.

Moreover, since $\epsglob(m)\leq \e(m)$, where $\e(m)$ is the stability threshold in Theorem~\ref{thm:stability}, we also have that $\epsglob(m)\to0$ as $m\to0^+$.
\end{proof}

\begin{remark}
While the ball of volume $V$ is not a minimizer of the nonlocal interaction energy $\nla+\nlb$ for any $V$ when $\alpha\in[d-1,d)$ (see \cite[Remark 3.2]{FraLie21}), the result in \cite{FFMMM} covers the full-range of $\alpha$-values; hence, the unit ball is also a minimizer of the energy $\enm$ for any $\e$ and for $m$ sufficiently small when $\alpha\in[d-1,d)$.
\end{remark}


\section{Numerical Insights}\label{sec:numerics}


In this final section we discuss the difficulties of obtaining general properties of the function $\e(m)$ when $0 < \beta <\beta_*$, and provide some numerical insight. At the root of these difficulties lies the fact that, in the case $0<\beta<\beta_*$, the supremum of the positive roots of the functions $\e_k$, i.e. $m_*=\sup_{k \geq 2} m_k^0$, is attained in the limit as $k\to\infty$ (see Remark~\ref{rmk:mstar2}). This implies that each curve $\e_k(m)$ crosses the $m$-axis at a point before $m_*$, that is, $m_k^0 < m_*$ for all $k\in\N$. As a consequence the value $\bar{k}(m)$ in Proposition~\ref{prop:eps-reg} where $\sup_{k\geq 1}\e_k(m)$ is attained switches infinitely many times. In other words, the curve $\e(m)$ is determined by an infinite number of cascading curves $\e_k(m)$ (see Figure~\ref{fig:betasmall}). More precisely, we claim that there exists values $m_2^c < \bar{m}_1 < \cdots <\bar{m}_j<\cdots<m_*$ with $\bar{m}_j \to m_*$ and an increasing subsequence $(k_j)_j$ such that
    \[
        \e(m) = \e_{k_j}(m) \qquad \text{for }  m \in (\bar{m}_j,\bar{m}_{j+1}).
    \]
Numerical plots indicate that for small values of $\beta$ \emph{all} curves $\e_k(m)$ may be active in defining the curve $\e(m)$ as depicted in Figure~\ref{fig:betasmall} (left), while as $\beta$ approaches $\beta_*$, some functions $\e_k$ are ``skipped'' in forming $\e(m)$ as in Figure~\ref{fig:betasmall} (right). This is not unexpected since when $\beta \geq \beta_*$ only the curves $\e_2$ and $\e_3$ are used in the profile of $\e$ while all other $\e_k$ with $k\geq 4$ are ``skipped''.

\begin{figure}
	\begin{center}
		\includegraphics[width=6.5cm]{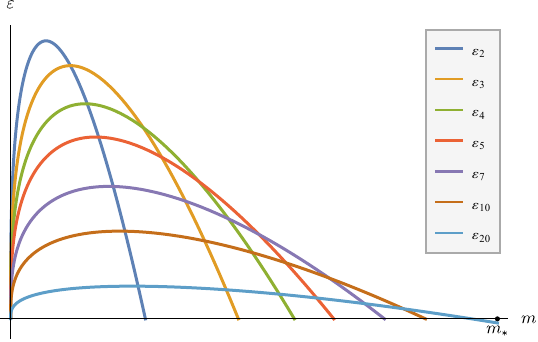}
		\hspace{1cm}
		\includegraphics[width=6.5cm]{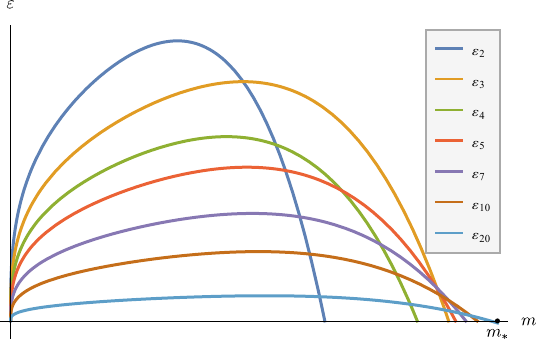}
	\end{center}
	\caption{Numerical plots of some of the functions $\e_k(m)$ in the case $\beta<\beta_*$. Left: $d=12$, $\alpha=9$, $\beta=4$. Right: $d=12$, $\alpha=9$, $\beta=40$. Here $\beta_*=86.5$. The picture on the right shows that some of the functions $\e_k$ can be ``skipped'' in the computation of the final profile $\e(m)$.}
	\label{fig:betasmall}
\end{figure}

The numerics suggest that, as in the case $\beta \geq \beta_*$, also when $0<\beta<\beta_*$ the function $\e(m)$ coincides with $\e_2(m)$ for $m\in[0,m_{2,3}]$. By Lemma~\ref{lemma:eps_2eps_3} and Lemma~\ref{lemma:A_k} it would suffice to show that $m_{3,k}>m_3^c$ to establish this claim. We further conjecture that $\e(m)$ is increasing for $m\in[0,m_2^c]$ and has a global maximum at $m_2^c$, where $m_2^c$ is the maximum point of the curve $\e_2$ (see Lemma~\ref{lem:basic_prop_eps_k}). In particular, numerics indicate that $\max_{m\in(0,m_*)} \e_k(m)$ is a decreasing sequence in $k$ for any $\beta>0$. 

Finally, we conjecture that the stability separation curve $\e(m)$ is decreasing on $(m_2^c,m_*)$. In order to establish this, it suffices to show one of the following two equivalent claims: (i) the curve $\e_{k+1}$ is decreasing at the intersection point of the curves $\e_k$ and $\e_{k+1}$ if it exists, i.e. $\e_{k+1}^\pr(m_{k,k+1})\leq 0$ for all $k\geq 3$, or (ii) the curves $\e_k$ and $\e_{k+1}$ intersect past the positive critical point of $\e_{k+1}$, i.e. $m_{k,k+1} \geq m_{k+1}^c$ for all $k\geq 3$. Both claims require one to establish that
\begin{equation} \label{eq:lambdak}
    \Lambda_k \defeq (d+\beta+1)\frac{A_k}{A_{k+1}} - (d-\alpha+1)\frac{B_k}{B_{k+1}} - (\alpha+\beta) > 0
\end{equation}
under the assumption $B_{k+1}<B_k$. Indeed, if $B_{k+1}>B_k$, then $\e_{k+1}(m)<\e_k(m)$ for all $m>0$ and $m_{k,k+1}$ does not exist. Proving \eqref{eq:lambdak}, however, does not seem to be an easy task, since numerics suggest that the monotonicity of this sequence changes as $\beta$ varies on $(0,\beta_*)$ (see Figure~\ref{fig:sequences}). 

\begin{figure}
	\begin{center}
		\includegraphics[width=6.5cm]{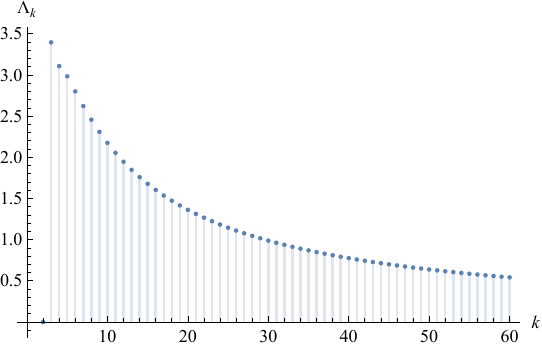}
		\hspace{1cm}
		\includegraphics[width=6.5cm]{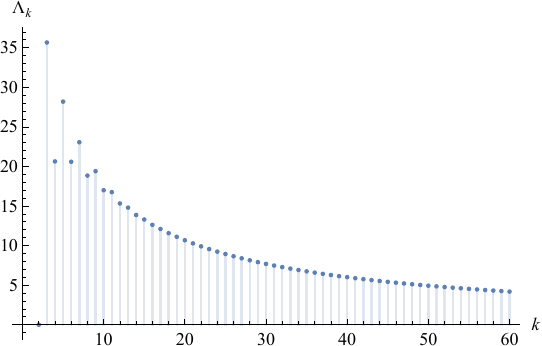}
	\end{center}
	\caption{Numerical plots of the sequence $\Lambda_k$ in \eqref{eq:lambdak}, whose positivity implies the inequality $m_{k,k+1}>m_{k+1}^c$. Left: $d=10$, $\alpha=8$, $\beta=10$. Right: $d=10$, $\alpha=8$, $\beta=130$. Here $\beta_*=134.$}
	\label{fig:sequences}
\end{figure}


 \bigskip
 \subsection*{Acknowledgments}
 MB is member of the GNAMPA group of INdAM.
 IT was partially supported by the Simons Foundation (851065), by the NSF-DMS (2306962), and by the Alexander von Humboldt Foundation.

\bibliographystyle{alpha}
\bibliography{references}

\end{document}